\documentclass[11pt]{amsart}
\usepackage{amssymb, amscd, verbatim, amsmath, MnSymbol}
\usepackage[center]{caption}
\newtheorem{thm}{Theorem}[section]
\newtheorem{lem}{Lemma}[section]

\newtheorem{rem}{Remark}[section]

\numberwithin{equation}{section}

\def\tr{\Delta}
\def\bea{\begin{eqnarray*}}
\def\eea{\end{eqnarray*}}
\def\be{\begin{eqnarray}}
\def\ee{\end{eqnarray}}

\def\tr{\Delta}

\def\de{\delta}

\def\({\left(}
\def\){\right)}

\def\e{\epsilon}

\def\na{\nabla}

\def\dis{\displaystyle}

\newcommand{\ric}{{\rm Ric}}
\begin{document}

\title{Rigidity of positively curved vacuum Static spaces}
\author[G. Yun]{Gabjin Yun}
\author[S. Hwang]{Seungsu Hwang} 
\address{Gabjin Yun\\Department of Mathematics, Myong Ji University\\ 
116 Myongji-ro Cheoin-gu, Yongjin, Gyeonggi 17058, Korea}
\email{gabjin@mju.ac.kr}
\address{Seungsu Hwang\\
Department of Mathematics, Chung-Ang University\\ 
84 Heukseok-ro, Dongjak-gu, Seoul 06974, Korea }
\email{seungsu@cau.ac.kr}
\maketitle

\begin{abstract}
In this paper, we study  compact vacuum static spaces of dimension $\ge 4$ with nonnegative sectional curvature. 
First, we prove that an $n$-dimensional  compact vacuum static space   is isometric to a standard sphere provided it has
positive sectional curvature  and the complete divergence of its Weyl curvature tensor is nonnegative.
Under the slightly weaker condition  of nonnegative sectional curvature,  we  show that such a vacuum static space must have parallel Ricci curvature tensor.  Second, we examine the critical point equation  arising from the critical metrics of the total scalar curvature functional on the space of Riemannian metrics restricted to unit volume and constant scalar curvature. 
We demonstrate  that a compact Riemannian manifold with nonnegative  sectional curvature and 
nonnegative complete divergence of the Weyl curvature tensor, which admits a nontrivial solution to the critical point equation,
 must be Einstein and  is isometric to a standard sphere. Our results can be viewed as extensions of the results in \cite{he}.
 \end{abstract}

\maketitle

\section{Introduction}
 A Riemannian manifold $(M,g)$ of dimension $n\geq 3$ is called a vacuum static space if there is a nontrivial solution function $f\in \ker s_g'^*$. 
 Here,  $s_g'^*$ is the formal $L^2$-adjoint of the linearization of the scalar curvature operator $s_g'$ at $g$ 
 which is defined as 
\be
 s_g'^*(f)=Ddf  -(\tr f)g - f\, {\rm Ric},\label{eqn2026-4-23-1}
 \ee
where  $Ddf $ and $\tr f$ are the Hessian and the (negative) Laplacian of $f$, respectively, and ${\rm Ric}$ is the Ricci curvature tensor of $M$.  For more details, one can refer \cite{bes}.

Kobayashi  \cite{kob} and Lafontaine \cite{laf}  independently provided a classification of compact locally conformally flat vacuum static spaces. Obata  \cite{oba} proved that a compact static Einstein space is isometric to a standard sphere. Qing and Yuan \cite{qy} classified complete Bach-flat vacuum static spaces with compact level sets of $f$. Recently, Kim  \cite{kim} classified complete vacuum static spaces with harmonic curvature (cf. \cite{li}). Recall that a manifold has harmonic curvature if the divergence of the Riemannian curvature tensor $R$ vanishes, i.e.,  $\mbox{div}R=0$.
Furthermore, H. He \cite{he} proved that  a $3$-dimensional compact vacuum static space with nonnegative Ricci curvature is locally conformally flat. By the results of Kobayashi and Lafontaine  \cite{kob, laf}, such  spaces can be completely  classified. 

In this paper, we investgate  compact vacuum static spaces of dimension $\ge 4$ with nonnegative sectional curvature. 
Using  a technical estimate on certain curvature tensors due to Catino \cite{cat} under the assumption of nonnegative sectional curvature, and through an analysis of various  tensors  in the vacuum static equation, we establish rigidity results for 
 spaces where the  complete divergence of the Weyl curvature tensor ${\mathcal W}$ is nonnegative.

\begin{thm} \label{thm2026-4-24-11}
Let $(M^n, g, f) $ be an $n$-dimensional $(n\ge 4)$  compact vacuum static space with positive sectional curvature. If ${\rm div}^4\, {\mathcal W}\geq 0$, then $(M,g)$ is isometric to a standard sphere.
\end{thm}

It is noteworthy  that the condition of positive section curvature $(K_M>0)$ cannot be relaxed to $K_M\geq 0$. Indeed,  there 
exist  compact vacuum static spaces with nonnegative sectional curvature and  harmonic curvature that are not  standard spheres. For instance,  Kim and Shin showed that  the product space ${\mathbb S}^2 (\frac s6)\times {\mathbb S}^2 (\frac s3)$ with  $K_M \geq 0$ is a compact vacuum static space with  harmonic curvature \cite{ks}. For $K_M\geq 0$, we prove the following.

\begin{thm} \label{thm2026-4-24-12}
 Let $(M^n, g, f)$ be an $n$-dimensional $(n\ge 4)$  compact vacuum static space with nonnegative sectional curvature. 
 If ${\rm div}^4\, {\mathcal W}\geq 0$, then $(M,g)$ has parallel Ricci curvature tensor (i.e., $D{\rm Ric}= 0$), where $D$ denotes the Levi-Civita connection.
\end{thm}

Next, we consider compact Riemannian manifolds admitting a nontrivial solution to the critical point equation. 
Let $M$ be a compact  $n$-dimensional smooth manifold, and let ${\mathcal M}_1$ be the space of all smooth Riemannian metrics of a unit volume on $M$. The total scalar curvature functional ${\mathcal S}$ on ${\mathcal M}_1$ is defined by
$$
{\mathcal S}(g) = \int_M s_g\, dv_g.
$$
where $s_g$ is the scalar curvature of the metric $g \in {\mathcal M}_1$ and $dv_g$ denotes the volume form of $g$.
Critical points of ${\mathcal S}$ on ${\mathcal M}_1$ are known to be Einstein (cf. \cite{bes}).
Introducing a subset ${\mathcal C}$ of ${\mathcal M}_1$ consisting of metrics with a constant scalar curvature,
the Euler-Lagrange equation of ${\mathcal S}$ restricted to ${\mathcal C}$ becomes
$$
\mathring {{\rm Ric}}_g=s_g'^*(f),
$$
called the critical point equation (CPE for short). Here, $\mathring{{\rm Ric}}_g$ is the traceless Ricci tensor defined by 
$\mathring{{\rm Ric}}_g = {\rm Ric}_g - \frac{s_g}{n}g$. In view of (\ref{eqn2026-4-23-1}), the CPE becomes
\be
\mathring{{\rm Ric}}_g = D_gdf -(\tr_g f)g -f {\rm Ric}_g.  \label{cpe}
\ee
From the variational problem for curvature functional, the Besse conjecture \cite{bes} describes that a non-trivial solution of the CPE should be Einstein. By taking the trace of (\ref{cpe}), as in vacuum static equation case, we have
 \be 
 \tr _g f=-\frac{s_g}{n-1}f \label{cpet}
 \ee 
Using the Ricci curvature tensor ${\rm Ric}_g = \mathring{{\rm Ric}}_g +  \frac{S_g}{n}g$, (\ref{cpe}) can be rewritten as 
$$
 (1+f){\rm Ric}_g = Ddf + \frac{s_g(nf+n-1)}{n(n-1)}g.$$
Since $g$ is clearly Einstein when $f=0$, going forward, we assume that $f$ is not trivial. 
Signicant  progress has been made toward Besse conjecture. This conjecture holds if  the Riemannian manifold is locally conformally flat \cite{laf}. 
It is also known that  Besse conjecture holds if underlying manifold has  harmonic curvature \cite{y-c-h}, \cite{erra},  
or the metric is Bach-flat \cite{qy}. Recall that $(M, g)$ is said to be Bach-flat  if the Bach tensor vanishes. 
Baltazar  \cite{bal} further showed that the  conjecture holds if ${\rm div}^4 {\mathcal W} = 0$ and the radially Weyl curvature vanishes. Specifically, Baltazar  \cite{bal} proved that a $3$-dimensional CPE metric with nonnegative sectional curvature is isometric to a standard sphere. It has been generalized to nonnegative Ricci curvature by He \cite{he} in dimension $3$. It is well-known that the sectional curvature is determined by the Ricci curvature in dimension three. 
 It is natural to ask the rigidity of vacuum static spaces under nonnegative sectional  curvature condition.
In this paper, we consider a compact Riemannian $n$-manifold  $(M^n, g), n \ge 4,$ admitting a nontrivial solution $f$ to the CPE. In case of vanishing  complete divergence of Weyl curvature tensor, we extend these results to higher dimensions $(n \ge 4)$.

\begin{thm}  \label{thm2026-4-24-13}
Let $(g,f)$ be a nontrivial solution to the CPE on a  compact  $n$-dimensional  $(n\ge 4)$  smooth manifold $M^n$ with $K_M \geq 0$. If $\mbox{div}^4\, {\mathcal W}\geq 0$,  then $(M,g)$ Einstein and, consequently,  is isometric to a standard sphere.
\end{thm}

Unlike  vacuum static spaces,   it is known that product manifolds do not admit nontrivial solutions to the CPE, which allows for the stronger rigidity conclusion in Theorem~\ref{thm2026-4-24-13}.

Rigidity results under positive Ricci curvature conditions remain unknown, and establishing such rigidities would be of significant interest.

\vspace{.12in}
\noindent
{\bf Notations and Convention}
We occasionally use the notation $\langle \,\, ,\,\, \rangle$ for metric $g$ or inner product induced by $g$ on tensor spaces. For curvature tensors, we basically follow the convention in \cite{bes} so that the the Riemannian curvature tensor $R$ is defined by
$R(X, Y)Z =D_YD_XZ  - D_XD_YZ -D_{[Y, X]}Z$. 

\section{Preliminaries}

As a preliminary for subsequent use,  we first recall the definition and key properties of the Cotton tensor.
Let $(M^n, g)$ be an $n$-dimensional  Riemannian manifold equipped with the Levi-Civita  connection $D$. 
We begin by considering a differential operator acting on the space  of symmetric $2$-tensors. 
For any symmetric $2$-tensor $h$  on $M$, its  covariant differential $d^Dh$ is defined by
$$ 
d^Dh(X, Y, Z) = D_Xh(Y, Z) - D_Yh(X, Z) 
$$
 for any vectors $X, Y$ and $Z$. The {\it Cotton tensor}  $C \in\Gamma(\Lambda^2 M \otimes T^*M)$ is then defined as 
 \be 
 C = d^D \left({\rm Ric} - \frac{s}{2(n-1)} g\right)   = d^D{\rm Ric} - \frac{1}{2(n-1)} ds \wedge g. \label{eqn2017-4-1-1}
 \ee 
It is known that for $n=3$, $C=0$ if and only if $(M^3, g)$ is locally conformally flat. For  $n \ge 4$, the vanishing of the Weyl
curvature tensor $\mathcal W$ implies  $C =0$. Conversely $C=0$ is equivalent to the Weyl curvature tensor being harmonic 
(i.e, $\de \mathcal W = 0$) due to the identity \cite{bes}: 
\be 
\de \mathcal W = - \frac{n-3}{n-2}d^D \left({\rm Ric} - \frac{s}{2(n-1)} g\right) = - \frac{n-3}{n-2} C\label{eqn2016-12-3-16} 
\ee 
under the following identification
$$ 
\Gamma(T^*M\otimes \Lambda^2M) \equiv \Gamma(\Lambda^2M \otimes T^*M). 
$$
We say that  the complete divergence of Weyl curvature tensor is nonnegative if ${\rm div}^4 \mathcal W \geq 0$, or equivalently ${\rm div}^3 C \geq 0$.
In a local frame $\{e_i\}$, let $C(e_i, e_j, e_k) = C_{ijk}$. It follows from (\ref{eqn2017-4-1-1}) that 
$$
C_{ijk} = R_{jk;i}-R_{ik;j} - \frac{1}{2(n-1)} s_i \de_{jk} + \frac{1}{2(n-1)} s_j \de_{ik}, 
$$ 
where ${\rm Ric} (e_i, e_j) = R_{ij}, \, R_{ij;k} = D_{e_k}{{\rm Ric}}(e_i, e_j)$ and $ds(e_i) = s_i$.
From this, we observe that the Cotton tensor $C$ satisfies the following properties: first, the  cyclic summation of $C_{ijk}$ vanishes;
 second, $C_{ijk}$ is skew symmetric in its first two indices, which implies that
$C(X, X, \cdot) = 0$ for any vector  field $X$; third, and the trace of $C$ in any two indices vanishes (i.e., $C_{iij} = C_{ijj} = C_{iji} = 0$).

In both the vacuum static equation and  CPE cases, the scalar curvature $s$ is constant. Consequently, the Cotton tensor simplifies to 
$$
C = d^D {{\rm Ric}} = d^D {\mathring {{\rm Ric}}}
$$
and we have the following algebraic identity for its components:
$$
C_{ijk}R_{jk;i}  =  |C|^2 + C_{ijk}R_{ik;j} = |C|^2 - C_{jik}R_{ik;j}.
$$ 
Here, we follow the Einstein summation convention for indices. This shows that 
\bea
\frac{1}{2}|C|^2=   C_{ijk}R_{jk;i} = C(e_i, e_j, e_k)D_{e_i}{{\rm Ric}}(e_j, e_k).\label{eqn2017-4-2-2} 
\eea
Furthermore, a direct calculation yields 
\be 
\frac{1}{2}|C|^2 
&=& 
C_{ijk}R_{jk;i} = (C_{ijk}R_{jk})_{;i} - C_{ijk;i}R_{jk} \nonumber\\ 
&=&
(C_{ijk}R_{jk})_{;i} + \langle \de C, {{\rm Ric}}\rangle,\label{eqn2018-1-23-1} 
\ee
where the divergence of the Cotton tensor is defined as $\de C (X, Y) = -{\rm div}\, C(X, Y)  = -D_{e_i} C(e_i, X, Y)$.

\vspace{.12in}
Recall that we denote by $D$ the Levi-Civita connection on a Riemannian manifold $(M, g)$. The composition $D^*D$ of 
the connection and its adjoint  acting on  symmetric $2$-tensors is called the rough Laplacian given  by $D^*D = D_{e_i}D_{e_i} - D_{D_{e_i}e_i}$ which is also denoted by $\Delta$.

\begin{lem}[\cite{bes}] \label{lem2026-4-23-10}
Assume a Riemannian $n$-manifold $(M^n, g)$ has constant scalar curvature. Then
$$
\delta d^D {{\rm Ric}} = -D^*D {{\rm Ric}} + {{\rm Ric}} \circ {{\rm Ric}} - \mathring R({{\rm Ric}}).
$$
Here ${{\rm Ric}}\circ  {{\rm Ric}}$  and $\mathring{R}({\rm Ric})$  are defined as 
$$
{{\rm Ric}}\circ  {{\rm Ric}}(X,Y)=\sum_{i=1}^n {{\rm Ric}}(X, e_i){{\rm Ric}}(Y,e_i)
$$
and
$$
\mathring{R}({\rm Ric})(X,Y)=\sum_{i=1}^n {{\rm Ric}}(R(X,e_i)Y, e_i),
$$
respectively,  for an orthonormal frame $\{e_i\}_{i=1}^n$ and any vectors $X,Y$.
\end{lem}

In Lemma~\ref{lem2026-4-23-10}, we would like to point out the operator $d = - {\rm div}$ denotes  the negative divergence operator, 
rather than the adjoint operator $\de^D$ of $d^D$. Indeed, it is straightforward  to verify  that $\de^D = 2\de$ when acting on 
symmetric $2$-tensors.

\begin{lem} \label{lem2026-4-24-2}
Let $(M^n, g)$ be a Riemannian $n$-manifold with constant scalar curvature. Then
\be
\frac12 \Delta |{\rm Ric}|^2  =  E +\langle {\rm div}\, C, {\rm Ric} \rangle + |D{\rm Ric}|^2,\label{eqn2026-4-22-21-1}
\ee
where
$$
E:= \langle {{\rm Ric}}\circ  {{\rm Ric}}, {{\rm Ric}}\rangle - \langle \mathring{R}({{\rm Ric}}), {{\rm Ric}}\rangle.
$$
\end{lem}
\begin{proof}
Since the scalar curvature is constant, we have $d^D {\rm Ric} = C$ and so it follows from Lemma~\ref{lem2026-4-23-10} that
$$
D^*D{\rm Ric}  = {\rm Ric}\circ {\rm Ric} - {\mathring R}({\rm Ric}) +{\rm div}\, C.
$$
Thus,
\bea
\frac12 \Delta |{\rm Ric}|^2  &=& \langle D^*D{\rm Ric}, {\rm Ric}\rangle + |D{\rm Ric}|^2\\
&=&
E +\langle {\rm div}\, C, {\rm Ric} \rangle + |D{\rm Ric}|^2.
\eea
\end{proof}

\vspace{.12in}
In \cite{cat}, Catino established an insightful pointwise estimate for a specific tensor to prove  rigidity results for four-dimensional 
shrinking Ricci solitons with positive curvature. This estimate remains valid for  Riemannian manifolds with  nonnegative sectional curvature.


\begin{lem} [\cite{cat}]  \label{lemma2} 
Let $(M^n,g)$ be Riemannian manifold of dimension  $n\geq 3$. If the sectional curvature $K_M$ satisfies $K_M\geq \epsilon s$ for some 
$\epsilon\in {\mathbb R}$, then the following estimate hold 
$$
E=\langle {{\rm Ric}}\circ  {{\rm Ric}}, {{\rm Ric}}\rangle - \langle \mathring{R}({{\rm Ric}}), {{\rm Ric}}\rangle \ge n \e s |\mathring{{\rm Ric}}|^2.
$$
\end{lem}

By Lemma~\ref{lemma2}, $E \ge 0$ if the sectional curvature of $(M, g)$ is nonnegative. Furthermore, if $(M, g)$ is a compact Riemannian manifold with positive sectional curvature, then
\be
E \ge nK_0 |\mathring{{\rm Ric}}|^2, \label{eqn2026-4-24-17}
\ee
where $K_0 = \min_M K_M>0$ denotes the minimum sectional curvature. In particular, this inequality implies that if $E=0$ on  a compact Riemannian manifold $(M^n, g)$ of dimension $\ge 3$ with  positive sectional curvature, then $(M, g)$ must be an
Einstein manifold.


\section{Vacuum Static spaces}

In this section we consider compact vacuum static spaces of nonnegative sectional curvature.
Recall that the vacuum static equation  can be rewritten as 
\be 
f {{\rm Ric}} =Ddf +\frac {sf}{n-1}g \label{equation1}
\ee
since the potential function $f$ satisfies 
$$\tr f =-\frac {s}{n-1}f.$$
It is well known that the scalar curvature of the vacuum static space should be constant. 

We begin by applying the covariant differential $d^D$ to both sides of (\ref{equation1}). Since the scalar curvature $s$ is constant, we have $d^D{{\rm Ric}} = C$, where $C$ denotes the Cotton tensor. Furthermore, it is  well known that $d^DDdf = {\tilde i}_{\na f}R$, where
the interior product ${\tilde i}_{\na f}R$ is defined by
$$
{\tilde i}_{\na f}R(X, Y, Z) = R(X, Y, Z, \na f)
$$ 
for any vectors $X, Y$ and $Z$. Thus,  taking the covariant differential $d^D$  of (\ref{equation1}) yields
\be 
df\wedge {{\rm Ric}} + fC= \tilde{i}_{\nabla f}R+ \frac s{n-1} df\wedge g.\label{equation2}
\ee
In this context, for a $1$-form $\phi$ and a symmetric $2$-tensor $h$, $\phi\wedge h$ is defined as 
$$
\phi \wedge h(X,Y, Z)= \phi (X)h(Y,Z)-\phi (Y)h(X,Z).
$$
For clarify, we define two  interior products $i_X$ and  $\tilde{i}_X$ for any vector field $X$. These operators  act on  a tensor  by 
substituting  $X$ into the first  and the  last slots, respectively. For instance, when acting on the Cotton tensor $C$, we have
$$ 
i_{X}C(Y,Z)= C(X, Y,Z)
$$
and
$$
 \tilde{i}_{X}C(Y,Z)= C(Y,Z,X)
 $$
 for any vector fields $Y$ and $Z$.
 
 With these preliminaries, we establish the following integral identity.
 
 \begin{lem}\label{lem2026-4-22-28-1}
 Let $(M^n, g, f)$ be an $n$-dimensional $(n\ge 4)$ compact vacuum static space satisfying (\ref{equation1}). Then
$$
{\rm div}(i_{\na f}C) = - i_{\na f}{\rm div}\, C,
$$
\be
\int_M f\langle i_{\na f}C, {{\rm Ric}}\rangle \, dv_g= \int_M {\rm div}\, C(\na f, \na f)\, dv_g\label{eqn2026-4-22-10}
\ee
and
$$
\int_M f^2 \langle {\rm div} C, {{\rm Ric}}\rangle \, dv_g+ \int_M {\rm div}\, C(\na f, \na f) \, dv_g= - \int_M  f\langle {\rm div}^2 C, \na f\rangle \, dv_g.
$$
\end{lem}
\begin{proof}
The first identity follows directly from the definition of divergence. For the second identity,  since $\langle i_{\na f}C, g\rangle = 0$, we have
\bea
{\rm div}\left(i_{\na f}\, C(\na f)\right) &=& \langle {\rm div}(i_{\na f}C), \na f\rangle + \langle i_{\na f}C, Ddf\rangle\\
&=&
-\langle i_{\na f}{\rm div}\, C, \na f\rangle + f \langle i_{\na f}C, {{\rm Ric}}\rangle
\eea
and so
\bea
\int_M f\langle i_{\na f}C, {{\rm Ric}}\rangle \, dv_g= \int_M {\rm div}\, C(\na f, \na f)\, dv_g.
\eea
Finally, since $\langle {\rm div}\,C, g\rangle =0$, from the vacuum static equation, we have the following
\bea
\int_M f^2 \langle {\rm div}\,  C, {{\rm Ric}}\rangle \, dv_g
&=&
\int_M f\langle {\rm div} \, C, f{{\rm Ric}}\rangle \, dv_g = \int_M f\langle {\rm div}\,  C, Ddf\rangle\, dv_g\\
&=&
-\int_M \langle {\rm div}(f{\rm div}\, C), \na f\rangle\, dv_g\\
&=&
- \int_M \left( \langle i_{\na f}{\rm div}\, C, \na f\rangle + f\langle {\rm div}^2 C, \na f\rangle\right) \, dv_g\\
&=&
- \int_M {\rm div}\, C(\na f, \na f) \, dv_g- \int_M  f\langle {\rm div}^2 C, \na f\rangle\, dv_g.
\eea
\end{proof}

From Lemma~\ref{lem2026-4-22-28-1}, we  have
\be
\int_M f\langle i_{\na f}C, {{\rm Ric}}\rangle \, dv_g+ \int_M f^2 \langle {\rm div} \, C, {{\rm Ric}}\rangle \, dv_g = - \int_M  f\langle {\rm div}^2 C, \na f\rangle\, dv_g.\label{eqn2026-4-23-15}
\ee
By taking the divergence of both sides in (\ref{equation2}), we have the following.

 \begin{lem} \label{lem2026-4-24-5}
 Let $(M^n, g, f)$ be an $n$-dimensional  $(n\ge 4)$ compact vacuum static space satisfying (\ref{equation1}). Then
 \be
\int_M \left( 4 f\langle i_{\na f}C, {{\rm Ric}}\rangle +  f^2\langle {\rm div}\, C, {{\rm Ric}}\rangle \right)\, dv_g=  \int_M \left( 3f^2 E  + f^2 |D{\rm Ric}|^2\right) \, dv_g. \label{eqn2026-4-24-3}
\ee
 \end{lem}
 \begin{proof}

By computing the divergence of each term in (\ref{equation2}), and applying both the definition of divergence and the vacuum static equation, we obtain the following.
\begin{itemize}
\item $\dis{{\rm div}(df\wedge {{\rm Ric}}) = D_{\na f}{{\rm Ric}} - f {{\rm Ric}}\circ {{\rm Ric}}}$\,;
\item $\dis{{\rm div}(fC) = i_{\na f}C + f{\rm div}\, C}$\,;
\item $\dis{{\rm div}(\tilde i_{\na f} R) = -i_{\na f}C - \tilde i_{\na f}C -f{\mathring R}({{\rm Ric}}) + \frac{sf}{n-1}{\rm Ric}}$\,;
\item $\dis{{\rm div}\left(df\wedge g\right) =-f {{\rm Ric}}}$.
\end{itemize}
Here, in the third identity, we used the fact that the cyclic summation of $C$ vanishes. 
Substituting these into the equation obtained from (\ref{equation2})  by taking the divergence operator, and taking the inner product with 
${\rm Ric}$, we obtain
\be
2\langle i_{\na f}C, {{\rm Ric}}\rangle + f\langle {\rm div}C, {{\rm Ric}}\rangle  = fE - \langle D_{\na f}{{\rm Ric}},{{\rm Ric}}\rangle.
\label{eqn2026-4-22-3-1}
\ee
Here, we used the fact  $\langle {\tilde i}_{\na f}C, {{\rm Ric}}\rangle = 0$. Note that
$$
\int_M f \langle D_{\na f}{{\rm Ric}},{{\rm Ric}}\rangle \, dv_g=  \frac12 \int_M f \langle \na f, \na |{\rm Ric}|^2\rangle   \, dv_g= - \frac14 \int_M f^2 \Delta |{\rm Ric}|^2\, dv_g.
$$ 
Thus, multiplying  (\ref{eqn2026-4-22-3-1}) by $f$ and integrating it over $M$, we can get
\be
\int_M \left( 2 f\langle i_{\na f}C, {{\rm Ric}}\rangle + f^2\langle {\rm div}\, C, {{\rm Ric}}\rangle \right)\, dv_g= \int_M f^2 E\, dv_g + \frac14  \int_M f^2 \Delta |{\rm Ric}|^2\, dv_g.\label{eqn2026-4-24-1}
\ee
Finally, substituting (\ref{eqn2026-4-22-21-1}) into (\ref{eqn2026-4-24-1}), we obtain
\bea
\int_M\left( 2 f\langle i_{\na f}C, {{\rm Ric}}\rangle + \frac12 f^2\langle {\rm div}\, C, {{\rm Ric}}\rangle \right) \, dv_g= \frac 32 \int_M f^2 E  \, dv_g+\frac12 \int_M f^2 |D{\rm Ric}|^2\, dv_g. 
\eea
 \end{proof}

 \begin{lem} \label{lem2026-4-24-6}
Let $(M^n, g, f)$ be an $n$-dimensional  $(n\ge 4)$ compact vacuum static space satisfying (\ref{equation1}). Then
$$
\int_M \left( 2f \langle i_{\na f}C, {{\rm Ric}}\rangle + f^2 \langle {\rm div}\,  C, {{\rm Ric}}\rangle \right)\, dv_g= -\frac12 \int_M f^2|C|^2 \, dv_g.$$
\end{lem}
\begin{proof}
Let us introduce a local frame $\{e_i\}$  and denote $C_{ijk} = C(e_i, e_j, e_k)$. By (\ref{eqn2018-1-23-1}),  we have
$$
\langle {\rm div} \, C, {{\rm Ric}}\rangle = (C_{ijk}R_{jk})_{;i}-\frac{1}{2}|C|^2.
$$
Thus,
 \bea
 f^2 \langle {\rm div} \, C, {{\rm Ric}}\rangle &=&  f^2 (C_{ijk}R_{jk})_{;i} - \frac{1}{2}f^2|C|^2\\
 &=&
 (f^2C_{ijk}R_{jk})_{;i} - 2ff_i C_{ijk}R_{jk} - \frac{1}{2}f^2|C|^2\\
 &=&
 (f^2C_{ijk}R_{jk})_{;i}- 2f \langle i_{\na f}C, {{\rm Ric}}\rangle - \frac{1}{2}f^2|C|^2.
 \eea
Integrating this over $M$ and using the divergence theorem, we obtain
\bea
\int_M \left( 2f \langle i_{\na f}C, {{\rm Ric}}\rangle + f^2 \langle {\rm div} C, {{\rm Ric}}\rangle \right) \, dv_g= -\frac 12 \int_M f^2|C|^2\, dv_g.\label{eqn2026-4-22-9-1}
\eea
\end{proof}

\begin{lem} \label{lem2026-4-24-9}
Let $(M^n, g, f)$ be an $n$-dimensional  $(n\ge 4)$  compact vacuum static space satisfying (\ref{equation1}). Then
$$
\int_M \left( \frac 32 f^2 E  + \frac12 f^2 |D{\rm Ric}|^2 + \frac 34 f^2|C|^2 \right) \, dv_g= \int_M  f\langle {\rm div}^2 C, \na f\rangle\, dv_g.$$
\end{lem}
\begin{proof}
By Lemma~\ref{lem2026-4-24-5} and Lemma~\ref{lem2026-4-24-6}, we have
\bea
- \int_M f^2 \langle {\rm div} \, C, {{\rm Ric}}\rangle \, dv_g= \int_M \left( 3f^2 E  + f^2 |D{\rm Ric}|^2 +f^2|C|^2\right)\, dv_g. \label{eqn2026-4-24-7}
\eea
and
\bea
\int_M f \langle i_{\na f}C, {{\rm Ric}}\rangle \, dv_g= \int_M \left( \frac 32 f^2 E  + \frac 12 f^2 |D{\rm Ric}|^2 + \frac 14 f^2|C|^2\right)\, dv_g. \label{eqn2026-4-24-8}
\eea
Subtracting these two, we obtain
\bea
\int_M\left( \frac 32 f^2 E  + \frac 12 f^2 |D{\rm Ric}|^2 + \frac 34 f^2|C|^2 \right)\, dv_g= -\int_M\left( f \langle i_{\na f}C, {{\rm Ric}}\rangle +f^2 \langle {\rm div} C, {{\rm Ric}}\rangle\right)\, dv_g.
  \eea
 Thus, from the third identity in Lemma~\ref{lem2026-4-22-28-1},
 \bea
\int_M\left( \frac 32 f^2 E  + \frac12 f^2 |D{\rm Ric}|^2 + \frac 34 f^2|C|^2\right)\, dv_g = \int_M  f\langle {\rm div}^2 C, \na f\rangle \, dv_g. \label{eqn2026-4-24-10}
 \eea
\end{proof}

\begin{lem}\label{lem2026-4-22-35-1}
We have the following
$$
\int_M  f\langle {\rm div}^2 C, \na f\rangle \, dv_g= - \frac12 \int_M f^2 {\rm div}^3 C\, dv_g.$$
\end{lem}
\begin{proof}
Since
\bea
{\rm div}(f^2 {\rm div}^2 C) &=& \langle {\rm div}^2 C, \na f^2\rangle + f^2 {\rm div}^3 C\\
&=&
2f\langle {\rm div}^2 C, \na f\rangle + f^2 {\rm div}^3 C,
\eea
we get, by integrating,
\bea
\int_M f \langle {\rm div}^2 C, \na f\rangle \, dv_g= - \frac12 \int_M f^2 {\rm div}^3 C\, dv_g.
\eea

\end{proof}

Now we are in position to prove Theorem~\ref{thm2026-4-24-11} and Theorem~\ref{thm2026-4-24-12}.

\vspace{.1in}
\noindent
{\bf Proof of Theorem~\ref{thm2026-4-24-11} and Theorem~\ref{thm2026-4-24-12}.}
By Lemma~\ref{lem2026-4-24-9} and Lemma~\ref{lem2026-4-22-35-1}, we have
\be
\int_M \left( \frac 32 f^2 E  + \frac 12 f^2 |D{\rm Ric}|^2 + \frac 34 f^2|C|^2 \right) \, dv_g= - \frac 12 \int_M f^2 {\rm div}^3 C \, dv_g. \label{eqn2026-4-24-15}
\ee
Since ${\rm div}^4 \mathcal W \geq 0$, or equivalently ${\rm div}^3 C \geq 0$, it follows from (\ref{eqn2026-4-24-15}) that
$$ 
E = 0\quad \mbox{and}\quad D{{\rm Ric}} = 0.
$$
In particular, if $K_M >0$, it follows from (\ref{eqn2026-4-24-17}) that
$\mathring{{\rm Ric}} = 0$, which implies $(M, g)$ is Einstein. Finally, by Obata's result \cite{oba}, $(M, g)$ is isometric to  a standard sphere.
\hfill{\large{$\Box$}}

\begin{rem} 
{\rm
We would like to point out that nonnegative sectional curvature condition $K_M \ge 0$ does not imply vanishing of $\mathring{{\rm Ric}}$.
Let $\{e_i\}_{i=1}^n$  be the eigenvectors of $\mathring{{\rm Ric}}$ and let $\lambda_i$ be the corresponding eigenvalues. Moreover, let $K_{ij}$ be the sectional curvature of the two-plane spanned by $e_i$ and $e_j$.  In \cite{cat}, Catino showed that 
$$
-E=\langle \mathring{R}({{\rm Ric}}), {{\rm Ric}}\rangle - \langle {{\rm Ric}}\circ {{\rm Ric}}, {{\rm Ric}}\rangle  \, \leq  -\sum_{i<j}(\lambda_i -\lambda_j)^2 K_{ij}\leq 0.
$$
Thus, if $K_{ij} = 0$ for some $i, j (i \ne j)$, then  $\lambda_i=\lambda_j$ cannot be guaranteed.
}
\end{rem}

\section{Critical Point Metrics}

In this section, we consider critical point metrics. Recall that a CPE metric is a triple $(M^n, g, f)$ for $n \ge 3,$ where $(M, g)$ is a compact (without boundary), oriented Riemannian manifold with positive constant scalar curvature and $f: M \to {\mathbb R}$ is a nonconstant smooth function satisfying the following critical point equation (CPE for short)
\be
(1+f){{\rm Ric}}=Ddf+\frac{s(nf +(n-1))}{n(n-1)}g. \label{eqn2026-4-24-20}
\ee
Comparing (\ref{eqn2026-4-24-20}) with the vacuum static equation (\ref{equation1}), one observes that they are structurally very 
 similar, with the potential function $f$ in the vacuum static case corresponding to $1+f$ in the CPE case. This similarity implies that 
the identities derived for vacuum static spaces in Section 3 remain valid for the CPE case by replacing
 $f$ by $1+f$. 
From this reason,  we briefly present  results  necessary for the proof of  Theorem~\ref{thm2026-4-24-13}, omitting 
detail computations.

First, since the scalar curvature $s$ is  constant for CPE metrics, we have 
$$
d^D{{\rm Ric}} = d^D{\mathring {{\rm Ric}}} = C.
$$
Taking the covariant differential  $d^D$ of (\ref{eqn2026-4-24-20}), we obtain
\be
df \wedge r + (1+f)C  = \tilde i_{\na f}R + \frac{s}{n-1} df \wedge g.\label{eqn2026-4-22-2}
\ee
Taking the divergence of each term in (\ref{eqn2026-4-22-2}), we get the following.
\begin{itemize}
\item $\dis{{\rm div}(df\wedge r) = D_{\na f}r - (1+f)r\circ r + \frac{s}{n}r} \,;$
\item $\dis{{\rm div}((1+f)C) = i_{\na f}C + (1+f){\rm div}C} \,;$
\item $\dis{{\rm div}(\tilde i_{\na f} R) = -i_{\na f}C - \tilde i_{\na f}C -(1+f){\mathring R}(r)  +\frac{s[nf+(n-1)]}{n(n-1)} r} \,;$
\item $\dis{{\rm div}\left(\frac{s}{n-1}df\wedge g\right) =-\frac{s}{n-1}(1+f)r + \frac{s^2}{n(n-1)}g}.$
\end{itemize}
Substituting these into the equation obtained from (\ref{eqn2026-4-22-2})  by taking the divergence operator, and taking the inner product with 
${\rm Ric}$, we obtain
\be
&&2\langle i_{\na f}C, {{\rm Ric}}\rangle + (1+f)\langle {\rm div}C, {{\rm Ric}}\rangle \nonumber\\
&&\qquad =
(1+f)E - \langle D_{\na f}{{\rm Ric}}, {{\rm Ric}}\rangle - \frac{s}{n-1}|{{\rm Ric}}|^2 + \frac{s^3}{n(n-1)}.\label{eqn2026-4-22-3}
\ee
Here, recall that
$$
E: = \langle {{\rm Ric}}\circ {{\rm Ric}}, {{\rm Ric}}\rangle - \langle{\mathring R}({{\rm Ric}}), {{\rm Ric}}\rangle.
$$
Multiplying (\ref{eqn2026-4-22-3}) by $1+f$, we have
\be
&&2(1+f)\langle i_{\na f}C, {{\rm Ric}}\rangle + (1+f)^2\langle {\rm div}\, C, {{\rm Ric}}\rangle \label{eqn2026-4-22-4}\\
&&\qquad =
(1+f)^2E - (1+f)\langle D_{\na f}{{\rm Ric}}, {\ric}\rangle\nonumber \\
& &\qquad - \frac{s}{n-1}(1+f)|{{\rm Ric}}|^2 + \frac{s^3}{n(n-1)}(1+f). \nonumber
\ee
One  difference between  the CPE and the vacuum static equation is the following. 
Since ${\rm div}(i_{\na f} \mathring {{\rm Ric}}) = (1+f)|\mathring {\ric}|^2$, we have
$$
 \int_M (1+f)|\mathring{\ric}|^2\, dv_g = 0.
$$
Thus, recalling $\int_M f = 0$ and $|{\ric}|^2  = |\mathring {\ric}|^2 + \frac{s^2}{n}$,
we have
\bea
&&\int_M \left(- \frac{s}{n-1}(1+f)|{\ric}|^2 + \frac{s^3}{n(n-1)}(1+f)\right)\, dv_g\\
&&\qquad
=
\int_M \left(- \frac{s}{n-1}(1+f)|\mathring{\ric}|^2 - \frac{s^3}{n(n-1)} + \frac{s^3}{n(n-1)}(1+f)\right)\, dv_g\\
&&\qquad =
0.
\eea
Integrating (\ref{eqn2026-4-22-4}) over $M$ and using the comments shown above, we get
\bea
&&\int_M \left( 2(1+f)\langle i_{\na f}C, {\ric}\rangle + (1+f)^2\langle {\rm div}C, {\ric}\rangle\right) \, dv_g\nonumber\\
&& \qquad =  \int_M (1+f)^2E\, dv_g + \frac14 \int_M (1+f)^2  \Delta |{\ric}|^2\, dv_g.
\eea
By applying Lemma~\ref{lem2026-4-24-2}, we obtain the following identity,  which corresponds to Lemma~\ref{lem2026-4-24-5} in
the vacuum static case.

\begin{lem}\label{lem2026-4-25-1}
Let $(M^n, g, f)$ be a solution of the CPE. Then
\bea
 \lefteqn{\int_M \left( 4(1+f)\langle i_{\na f}C, {\ric}\rangle + (1+f)^2\langle {\rm div}\, C, {\ric}\rangle\right) \, dv_g}\\
  &=&
\int_M \left( 3(1+f)^2E +  (1+f)^2  |D {\ric}|^2\right) \, dv_g.
\eea
\end{lem}

The result for the CPE case corresponding to Lemma~\ref{lem2026-4-24-6}  is the following.

\begin{lem}\label{lem2026-4-25-2}
Let $(M^n, g, f)$ be a solution of the CPE. Then
$$
\int_M \left( 2(1+f) \langle i_{\na f}C, {\ric}\rangle + (1+f)^2 \langle {\rm div}\,  C, {\ric}\rangle \right) \, dv_g= -\frac12 \int_M (1+f)^2|C|^2 \, dv_g.$$
\end{lem}

Next, the result for the CPE case corresponding to Lemma~\ref{lem2026-4-24-9}  is the following.

\begin{lem}\label{lem2026-4-25-3}
Let $(M^n, g, f)$ be a solution of the CPE. Then
\begin{eqnarray}
 & & \int_M \left( \frac23 (1+f)^2 E + \frac12 (1+f)^2 |D\ric|^2 + \frac34 (1+f)^2 |C|^2\right) \, dv_g\label{eqn2026-4-22-9-10-10}\\
& & \qquad = \int_M (1+f)\langle {\rm div}^2C, \na f\rangle\, dv_g.\nonumber
\end{eqnarray}
\end{lem}
\begin{proof}
By Lemma~\ref{lem2026-4-25-1} and Lemma~\ref{lem2026-4-25-2}, we have
\bea
- \int_M (1+f)^2 \langle {\rm div}\,  C, {\ric}\rangle \, dv_g= \int_M \left( 3(1+f)^2 E  + (1+f)^2 |D\ric|^2 + (1+f)^2|C|^2\right) \, dv_g. \label{eqn2026-4-24-7-10}
\eea
and
\bea
\int_M (1+f) \langle i_{\na f}C, {\ric}\rangle\, dv_g = \int_M \left( \frac32 (1+f)^2 E  + \frac12 (1+f)^2 |D\ric|^2 + \frac14 (1+f)^2|C|^2\right)\, dv_g. \label{eqn2026-4-24-8-10}
\eea
Subtracting these two, we obtain
\bea
&&\int_M \left( \frac32 (1+f)^2 E  + \frac12 (1+f)^2 |D\ric|^2 + \frac34 (1+f)^2|C|^2 \right) \, dv_g\\
&&\qquad
= -\int_M \left( (1+f) \langle i_{\na f}C, {\ric}\rangle +(1+f)^2 \langle {\rm div}\, C, {\ric}\rangle\right) \, dv_g.
  \eea
 Thus, the equation (\ref{eqn2026-4-22-9-10-10}) follows from Lemma~\ref{lem2026-4-22-287} below.
\end{proof}

\begin{lem}\label{lem2026-4-22-287}
We have
$$
{\rm div}(i_{\na f}C) = - i_{\na f}{\rm div}C,
$$
$$
\int_M (1+f)\langle i_{\na f}C, {\ric}\rangle \, dv_g= \int_M {\rm div}\, C(\na f, \na f)\, dv_g,$$
and
$$
\int_M \left((1+f)^2 \langle {\rm div}\,  C, {\ric}\rangle + {\rm div}\, C(\na f, \na f) \right)\, dv_g = - \int_M  (1+f)\langle {\rm div}^2 C, \na f\rangle\, dv_g.
$$
\end{lem}

\begin{lem}\label{lem2026-4-22-35}
We have the following
$$
\int_M  (1+f)\langle {\rm div}^2 C, \na f\rangle\, dv_g = - \frac12 \int_M (1+f)^2 {\rm div}^3 C\, dv_g.
$$
\end{lem}
\begin{proof}
Since
\bea
{\rm div}((1+f)^2 {\rm div}^2 C) &=& \langle {\rm div}^2 C, \na (1+f)^2\rangle + (1+f)^2 {\rm div}^3 C\\
&=&
2(1+f)\langle {\rm div}^2 C, \na f\rangle + (1+f)^2 {\rm div}^3 C,
\eea
we get, by integrating,
$$
\int_M (1+f) \langle {\rm div}^2 C, \na f\rangle\, dv_g = - \frac12 \int_M (1+f)^2 {\rm div}^3 C\, dv_g.$$

\end{proof}

\vspace{.12in}
\noindent
{\bf Proof of Theorem~\ref{thm2026-4-24-13}}\, 
From Lemma~\ref{lem2026-4-25-3} and Lemma~\ref{lem2026-4-22-35}, we obtain
\bea
\lefteqn{\int_M \left( \frac32 (1+f)^2E + \frac12 (1+f)^2 |DRic|^2 +  \frac34 (1+f)^2 |C|^2 \right) \, dv_g}\\
&=& - \frac12\int_M (1+f)^2 {\rm div}^3 C\, dv_g.
\eea
Since ${\rm div}^4 \mathcal W \geq 0$ and ${\rm div}^3 C \geq 0$ with $E\ge 0$, we have
$$
C = 0.
$$
In other words, the CPE metric has harmonic curvature (i.e., ${\rm div} R =   d^D {\ric} = C =0)$. Applying the main result in 
 \cite{y-c-h, erra}, we can conclude  that  $(M, g)$ is isometric to a standard sphere. 
\hfill{\large $\Box$}

\end{document}